\documentclass[10pt]{article}
\usepackage{hyperref}
\usepackage{amssymb, amsmath, amsthm}
\usepackage{url}
\usepackage{enumitem}
\usepackage{geometry}
\usepackage{mathrsfs}
\usepackage{setspace}
\usepackage[pdftex]{graphicx}
\usepackage{makecell,xcolor}
\usepackage{tikz}
\usepackage{dsfont}
\usepackage[all,cmtip]{xy}
\let\svthefootnote\thefootnote
\newcommand\freefootnote[1]{%
  \let\thefootnote\relax%
  \footnotetext{#1}%
  \let\thefootnote\svthefootnote%
}

\makeatletter

\newcommand{\address}[1]{\gdef\@address{#1}}
\newcommand{\email}[1]{\gdef\@email{\url{#1}}}

\newcommand{\@endstuff}{\par\vspace{\baselineskip}\noindent\small
\begin{tabular}{@{}l}\scshape\@address\\\textit{E-mail address:} \@email\end{tabular}}
\AtEndDocument{\@endstuff}
\makeatother

\address{\parbox{\linewidth}{St. Petersburg Departement of Steklov Math. Institute,\\ Fontanka 27, 191023 St. Petersburg, Russia}}
\email{polyakov@pdmi.ras.ru}

\newtheorem{theorem}{Theorem}

\newtheorem{lemma}{Lemma}

\theoremstyle{definition}

\theoremstyle{remark}

\DeclareMathOperator{\Gal}{Gal}
\DeclareMathOperator{\Tr}{Tr}
\DeclareMathOperator{\res}{res}

\title{Artin-Hasse formula for $p^m-$primary elements

}
\author{Vladimir Polyakov}
\date{}

\begin{document}

	\maketitle
	
	\freefootnote{This work was supported by the Ministry of Science and Higher Education of the Russian
Federation, agreement № 075-15-2022-289}
	\begin{abstract}
	    Using Borevich's system of generators and relations, the classical Artin-Hasse formula is obtained from scratch in the case when taking $p$-th root of the second argument of the Hilbert symbol gives an unramified $p$-extension of the same degree of irregularity. Under the same assumptions, in the case of Lubin-Tate formal groups, an expression for the Hilbert symbol is obtained in terms of the expansion of elements by the same system of generators.
	\end{abstract}
	\abovedisplayskip=.8\abovedisplayskip
\abovedisplayshortskip=.8\abovedisplayshortskip
\belowdisplayskip=.8\belowdisplayskip
\belowdisplayshortskip=.8\belowdisplayshortskip
\section*{Introduction}

The classical Hilbert symbol for a one-dimensional local field $K$ containing the $p^n$-th root of unity $\zeta_{p^n}$ is the following pairing
$$(-,-)_{p^n}:K^\times\times K^\times\rightarrow \left<\zeta_{p^n}\right>,\ (a,b)_{p^n}=\frac{\sigma_a(\sqrt[p^n]{b})}{\sqrt[p^n]{b}}$$
where $\sigma_{()}:L^\times\rightarrow\Gal(L^{ab}/L)$ is the local reciprocity map. There are two types of explicit formulas for the Hilbert symbol: Kummer type formulas and Artin-Hasse type formulas.

Formulas of the first type are generalizations of the classical Kummer formula \cite{Kummer}:
$$(a,b)_p=\zeta_p^\gamma,\ \gamma=\res(\log b(X)d\log a(X)X^{-p})$$
where $a,b$ are the principal units of the field $K=\mathbb{Q}_p(\zeta_p)$ and $a(X), b(X)$ are power series with coefficients from $\mathbb{Z}_p$ such that $a(X)|_{X=\zeta_p-1}=a,\ b(X)|_{X=\zeta_p-1}=b$. Formulas of this type include the formulas of Iwasawa and Wiles. One of the most general formulas of this type is the Bruckner-Vostokov formula \cite{VostokovGlavn1}, \cite{VostokovGlavn2}, \cite{Bruckner}.

Classical formulas of the Artin-Hasse type do not require the construction of power series and taking derivatives and residues \cite{ArtinHasse}:
$$(\zeta_{p^n},b)_{p^n}=\zeta_{p^n}^\gamma, \gamma=\frac{1}{p^n}\Tr_{K/\mathbb{Q}_p}(\log b)$$
$$(\pi,b)_{p^n}=\zeta_{p^n}^\gamma, \gamma=\frac{1}{p^n}\Tr_{K/\mathbb{Q}_p}(\zeta\pi^{-1}\log b)$$
where $K=\mathbb{Q}_p(\zeta_{p^n})$, $\pi=\zeta_{p^n}-1$. However, power series already appear in generalizations of these formulas (Sh. Sen \cite{ShSen}):
$$(a,b)_{p^n}=\zeta_{p^n}^\gamma, \gamma=\frac{1}{p^n}\Tr_{K/\mathbb{Q}_p}\left(\frac{\zeta_{p^n}}{g'(\pi)}\frac{f'(\pi)}{f(\pi)}\log b)\right)$$
where $g(\pi)=\zeta_{p^n}, f(\pi)=a$, plus some conditions on $a,b$.

In this paper we will prove the second Artin-Hasse formula, but only for those $b$ whose $p-$th roots added to the field $\mathbb{Q}_p$ give an unramified $p-$extension with the same degree of irregularity. The proof will not follow the classical path. Usually, a formula is first guessed, after which it is checked on some convenient system of generators, and then it is concluded that this formula is true for all elements. We will not initially assume that the formula is known, but by successive transformations we will derive it from scratch, so that it will be clear where it comes from. To do this, we turn to a non-standard system of generators and relations for these purposes, which was constructed by Borevich \cite{Borevic} and Iwasawa \cite{Iwas} for the classical case of an unramified $p-$extension of a local field and further generalized to the case of formal modules over unramified extensions \cite{Ilya}, \cite{Tigran}, \cite{Polyakov}.

Borevich's article also gives other similar constructions for cyclic extensions that can be generalized to the case of formal groups. Apparently, such generalizations are limited to the above list of papers on unramified extensions. Then, if similar systems of generators and relations for formal modules over extensions with ramification appear, then having done similar calculations, it will be possible to extend the results of this paper to these cases.

This paper presents the simplest case of using this system of generators. By longer calculations, the same Artin-Hasse formula is obtained for the case of the higher cyclotomic extensions ($n>1$). Apparently, this method is applicable to Lubin-Tate formal groups and more complex formal groups, but this requires an explicit description of the generators, which may require resorting to explicit reduction theory and much more subtle calculations.
\section*{Preliminaries}
Let $K$ be a local field (finite extension of $\mathbb{Q}_p$) with residue field $k$, $\pi$ is a prime element in $K$, $F$ is a one dimensional formal group law over ring of integers $\mathcal{O}_K$ with maximal ideal $\mathfrak{m}_K$. Let us define the structure of the formal module on the maximal ideal $F(\mathfrak{m})$ in the standard way:
$$a+_F b=F(a,b),\ \alpha a=[\alpha](a)$$
where $a,b\in F(\mathfrak{m}),\ \alpha\in \mathcal{O}_{K_0}\simeq End_{\mathcal{O}_K}(F)$.

 Throughout this article, we will stick to the following assumptions: consider the following chain of extensions
$$\mathbb{Q}_p-K-L-M$$
where $L/K$ any finite extension of degree $n$, $M/L$ unramified extension of degree $p^m$ and $\mathfrak{m}$ is the maximal ideal of $\mathcal{O}_M$. $F$ is the Lubin-Tate group law on $K$. We also assume that fields $L$ and $M$ are $s-$irregular, i.e. contain $\Lambda_{\pi,s}$, but do not contain $\Lambda_{\pi,s+1}$, where $\Lambda_{\pi,k}=\ker[\pi^k]$. Also let $G=Gal(M/L)=\left<\sigma\right>$ be a Galois group of M over L with $\sigma$ generator. We will denote by $\zeta$ the generator of the $\Lambda_{\pi,s}$.
Consider the following perfect pairing
$$\frac{L^{\times }}{N_{M/L}M^{\times }}\times \frac{([\pi^s]F(\mathfrak{m}_M))\cap F(\mathfrak{m}_L)}{[\pi^s]F(\mathfrak{m}_L)}\rightarrow\Lambda_{\pi,s}$$
$$(t,x)\mapsto \sigma_t(y)-_F y$$
where $[\pi^s](y)=x$, $M=L(y)$ - unramified $p-$extension of field $L$, $\sigma_{(\cdot)}$ - reciprocity map. We call it generalised Hilbert pairing.
\\

Since we require that the adjunction $y$ to $L$ gives an unramified extension $M$, the reciprocity mapping will give us a Frobenius automorphism $\sigma^k\in G$ with $k$ equal to the valuation $\nu(t)$, where $\nu$ normalized such that $\nu(\pi_L)=1$. We will study this pairing in the simplest case, when $k=1$, the rest of the cases are derived in a similar way.

From \cite{Ilya},\cite{Tigran} and \cite{Polyakov} we know the following structure of $F(\mathfrak{m})$ as $\mathcal{O}_K[G]$-module. 
\begin{theorem}\label{osnov}
If the previous conditions are met, then for the $\mathcal{O}_K[G]$-module $F(\mathfrak{m})$ there is the following system of generators $\theta_1,...,\theta_{n-1},\xi,\omega$ with a single relation $(\sigma -_F 1)(\omega)=[\pi^s](\xi)$.
\end{theorem}

So we can express our $y$ as follows:
$$y=\sum_{i=1...n-1,\ j=1...p^m}^F [d_{ij}]\theta_i^{\sigma^j}+_F\sum^F_{j}([c_j]\omega^{\sigma^j}+_F[b_j]\xi^{\sigma^j})$$
where $d_{ij}, c_j, b_j\in\mathcal{O}_K$.\\
Since $\sigma y-_F y\in \Lambda_{\pi,s}$, $d_{ij}$ doesn't depend on $j$ and hence we can rename them as $d_i$:
\begin{equation}\label{equ1}
    y=\sum^F_{i} [d_{i}] N_F\theta_i+_F\sum^F_{j}([c_j]\omega^{\sigma^j}+_F[b_j]\xi^{\sigma^j})
\end{equation}
where $N_F(\theta)=\theta+_F\theta^\sigma+_F\theta^{\sigma^2}+_F...+_F\theta^{\sigma^{p^m-1}}$ is the formal norm operator.

Applying $\sigma -_F 1$ and using the relation from the theorem, we obtain the following
$$\sigma y -_F y=\sum^F_{i}[\pi^s c_i]\xi^{\sigma^i}+_F \sum^F_{i}[b_{i-1}-b_i]\xi^{\sigma^i}=\sum^F_{i}[\pi^s c_i+b_{i-1}-b_i]\xi^{\sigma^i}=[\gamma]\zeta$$
for some $\gamma\in\mathcal{O}_K$. In this equality $\gamma$ is the desired value of the "exponent" of $\zeta$ in the Hilbert symbol, it is determined up to $\pi^s$ and calculating it is our main task.
\subsection*{Expression of $\gamma$ in terms of the coefficients of the formal expansion of $y$}
Let us prove a small technical lemma.
\begin{lemma}
Suppose that $\sum^F_{i}[\delta_i]\xi^{\sigma^i}=0$. Then all $\delta_i$ are equal and lie in the ideal $(\pi^s)$.
\end{lemma}
\begin{proof}
From \cite{Polyakov} lemma 5 we have $\delta_i=\pi^s\delta_i^1$, then $[\pi^s]\sum^F[\delta_i^1]\xi^{\sigma^i}=0$. Therefore for some $\alpha^1$ we have the following $$\sum^F[\delta_i^1]\xi^{\sigma^i}=[\alpha^1]\zeta=\sum^F[\alpha^1]\xi^{\sigma^i}$$
If we move everything to the left hand side then again by lemma 5 we get $\delta_i^1-\alpha^1=\pi^s\delta_i^2$.
After that, repeating this procedure we will build new $\delta_i^k,\ \alpha^k$, and therefore if we let $k\rightarrow \infty$, then for any $i$ we get $\delta_i^k\rightarrow\delta$ for some $\delta\in\mathfrak{m}^s$. 
\end{proof}
Now we will use this lemma in the following manner. Put $\delta_i=\pi^s c_i+b_{i-1}-b_{i}-\gamma$. Then $$\pi^s c_i+b_{i-1}-b_{i}-\gamma=\pi^s\chi$$
If we sum up all these equalities by $i$, we get 
$$\pi^s\sum c_i -p^m\gamma=p^m\pi^s\chi$$
$\gamma$ is determined modulo $\pi^s$, therefore we get
$$\gamma\equiv\frac{\pi^s}{p^m}\sum c_i\mod{\pi^s}$$
\subsection*{The main equation}
Let us introduce the following notations: $\epsilon_i=N_F \theta_i\ (i=1,...,n-1),\ \Omega=N_F\omega$. 
Applying the formal norm operator $N_F$ to our last equality (\ref{equ1}), we get
$$[p^m]x=N_F x=N_F[\pi^s]y=\sum^F[\pi^s p^m d_i]\epsilon_i+_F[p^m\gamma]\Omega$$
Taking formal Lubin-Tate logarithm and dividing by $p^m$, we get our main equation
\begin{equation}\label{osn}
    \lambda x=\sum \pi^s d_i\lambda\epsilon_i+\gamma\lambda\Omega
\end{equation}
From this equality we need to get $\gamma$, assuming $x,\ \epsilon_i,\ \Omega$ are known, and $d_i$ assuming unknowns. Next, we will show how to get rid of unknowns in the simplest particular case and obtain the Artin-Hasse formula.

\subsection*{Getting rid of the unknowns}

Henceforth we will work with the simplest case: $K=\mathbb{Q}_p,\ s=1$, $L=K_1=K(\zeta)$, where $\zeta=\eta-1,\ \eta^p=1$. Again assume that $M/L$ is unramified extension of degree $p^m$ which does not contain roots of unity of the next degree $p$ with $G=\left<\sigma\right>$. Let $\pi_L=\zeta$ be a prime element of the field $L$. Note that in our case the Lubin-Tate formal group law $F$ coincides with the classical multiplicative formal group law.

We constructed the system of generators and relations of formal module $F(\mathfrak{m}_M)$ in Theorem \ref{osnov} as follows (see Theorem 4 in \cite{Polyakov}). First, we considered the system of generators $F(\mathfrak{m}_L)$, For this purpose, we fixed $\zeta$ as one of the generators and complemented it to a complete system of generators.
More specifically, we complemented $\zeta$ by exponents from the generators of $\mathcal{O}_L$ over $\mathcal{O}_K$ multiplied by $\Pi$ (In this case, all exponents converge, since the extension $L/K$ is totally ramified and the generators of $\mathcal{O}_L$ over $\mathcal{O}_K$ will simply be $1, \Pi,\Pi^2, \ ...,\Pi^{p-2}$).
Then the system of generators of $F(\mathfrak{m}_L)$ as $\mathcal{O}_K$-module will be as follows:
$$\Pi,\ Exp(\Pi^2),\ Exp(\Pi^3),\ ...,Exp(\Pi^{p-1})$$
where $N_F(\xi)=\zeta$.
Next, we took the preimages of these elements with respect to the action of $N_F$, after that, by triviality of $H^1(G,F(\mathfrak{m}_M))$ we got $\omega$ (which satisfies the relation $\omega^\sigma-_F\omega=[\pi^s]\xi$) and added this element to the rest of the generators. As a result, we received a complete system of generators of $F(\mathfrak{m}_M)$ as $\mathcal{O}_K[G]$-module  
$$\xi,\omega,\theta_1,\ ...,\theta_{p-2}$$
Thus, we get the following $$\lambda\epsilon_i=\lambda(N_F(\theta_i))=\lambda(N_FN_F^{-1}Exp{\Pi^{i+1}}))=\Pi^{i+1}$$
Hereinafter, we will denote $\Tr=\Tr_{L/K}$.
\begin{lemma}
Let $f$ be the minimal polynomial of $\Pi$ over $\mathbb{Q}_p$, then $\Tr(\frac{\Pi^i}{f'(\Pi)})=0$ for $i=0,\ ...,p-3$.
\end{lemma}
A proof of a similar lemma can be found in \cite{VostFes} p.70.
\\ \\
Thus, in our case, equality \ref{osn} will be as follows
$$\lambda x=p\sum_{i=1,\ ...,p-2} d_i\Pi^{i+1}+\gamma\lambda\Omega$$
Dividing this equality by $\Pi^2f'(\Pi)$ we get
$$\frac{1}{\Pi^2f'(\Pi)}\lambda x=p\sum_{i=1,\ ...,p-2} d_i\frac{\Pi^{i-1}}{f'(\Pi)}+\gamma\frac{\lambda\Omega}{\Pi^2f'(\Pi)}$$
Taking the trace operator and using Lemma 2, we get rid of $d_i$
$$\Tr(\frac{1}{\Pi^2f'(\Pi)}\lambda x)=\gamma \Tr(\frac{\lambda\Omega}{\Pi^2f'(\Pi)})$$
Because $f(T)=\frac{(1+T)^p-1}{T}$, then $\Pi^2f'(\Pi)=p\frac{\Pi}{\eta}$, and therefore
\begin{equation}\label{sec}
    \Tr(\Pi^{-1}\eta\lambda x)=\gamma \Tr(\Pi^{-1}\eta\lambda\Omega)
\end{equation}

Now, to get the final formula, we need to calculate the multiplier at $\gamma$. But since $\gamma$ is defined modulo $p$, it will be enough for us to calculate it modulo $p\Tr(\Pi^{-1}\eta\lambda\Omega)$.
\subsection*{Calculating $\lambda\Omega$}
Recall that we get $\omega$ as the element that gives us a coboundary:
$$\omega^\sigma-_F\omega=[p]\xi$$
From the proof of analogue of Hilbert's Theorem 90 for unramified extensions (\cite{Polyakov} Theorem 3) it is clear that
$$\Tr_{M/L}\lambda\omega=-\pi^s \Tr_{M/L}(\chi^\sigma\lambda\xi+\chi^{\sigma^2}(\lambda\xi+\lambda\xi^{\sigma})+\chi^{\sigma^3}(\lambda\xi+\lambda\xi^{\sigma}+\lambda\xi^{\sigma^2})+...+\chi^{\sigma^{p^m-1}}(\lambda\xi+\lambda\xi^{\sigma}+...+\lambda\xi^{\sigma^{p^m-2}}))$$
Where $\chi$ any element, such that $\Tr_{M/L}(\chi)=1$ (by surjectivity of the trace operator in the case of unramified extension), in our case $\pi^s$ equals $p$.

The element $\xi$ is defined by equality $N_F\xi=\Pi$ as before. Consider it modulo $\Pi^2$, then we get $\Tr_{M/L}\xi=\Pi \mod{\Pi^2}$. Therefore we get (for more details \cite{Ilya} Lemma 6)
$$\xi=\Pi\chi \mod{\Pi^2}$$

In our case, the formal logarithm has the following simplest form: $\lambda(X)=X+\frac{X^2}{2}+\frac{x^3}{3}+...$, therefore $$\lambda\xi=\Pi(\chi-\chi^p) \mod{\Pi^2}$$

Consider first the case $m=1$ and then reduce the general case to this.\\ \\
Since all our calculations are performed modulo $\Pi$, we can assume that we are working in the residue fields: $l=\mathbb{F}_p,\ m=\mathbb{F}_{p^p}$. For degree $p$ extensions of finite fields, it is known that there exists a self-adjoint normal basis $\tau, \tau^p,\ \tau^{p^2},\ ...,\tau^{p^{p-1}}$ of the field $\mathbb{F}_{p^p}$ over $\mathbb{F}_p$ for which the trace of each element of this basis is equal to 1.
As elements of this basis, we can take the roots of the following irreducible polynomial over $\mathbb{F}_{p^p}$ (\cite{Shuhong} Theroem 5.4.2):
$$x^p-x^{p-1}-1$$
Then as $\chi$ we can take an element $\tau$ of this basis (since by construction we can take any element with trace equals 1). Then, due to the self-adjointness of the basis, the following relations hold: 
$$\Tr(\chi^{2p^k})=1,\ \Tr(\chi^{p^k+p^j})=0,\ \forall k,j,\ j\neq k$$
Finally, we can compute $\Tr_{M/L}\lambda\omega \mod{p\Pi^2}$, taking as $\chi$ any its lifting from the residue field
$$\frac{1}{-p\Pi}\Tr_{M/L}\lambda\omega=\frac{1}{\Pi}\Tr(\chi^\sigma\lambda\xi+\chi^{\sigma^2}(\lambda\xi+\lambda\xi^{\sigma})+...+\chi^{\sigma^{p-1}}(\lambda\xi+\lambda\xi^{\sigma}+...+\lambda\xi^{\sigma^{p-2}}))=$$
$$=\Tr(\chi^\sigma(\chi-\chi^\sigma)+\chi^{\sigma^2}(\chi-\chi^\sigma+\chi^\sigma-\chi^{\sigma^2})+...+\chi^{\sigma^{p-1}}(\chi-\chi^{\sigma^{p-1}}))$$
Due to the self-adjointness of the basis inside the trace in each bracket, only the last term will give us a nonzero contribution, so we have the following:
$$\frac{1}{-p\Pi}\Tr_{M/L}\lambda\omega=\Tr(-\chi^{2p}-\chi^{2p^2}-...-\chi^{2p^{p-1}})=-p+1=1\mod{\Pi}$$
So we have $\lambda\Omega=\Tr_{M/L}\lambda\omega=-p\Pi+p\Pi^2\delta$, for some $\delta$, and then
$$\Tr_{L/K}\Pi^{-1}\eta\lambda\Omega=\Tr_{L/K}(-\Pi^{-1}\eta p\Pi+\Pi^{-1}\eta p\Pi^2\delta)=-p\Tr_{L/K}(\eta)+p^2\delta'=p+p^2\delta'$$ \\
Consider the case $m>1$. Consider an unramified extension of degree $p$: $H/L$. Then, by the transitivity of the trace: $\Tr_{M/L}=\Tr_{H/L}\circ \Tr_{M/H}$, and surjectivity for unramified extensions, as $\chi\in F(\mathfrak{M}_M)$ we can take an element that under action of trace operator, will go to $\chi$ from the previous case from the field $H$. After that, all subsequent calculations will be similar to the calculations from the previous case. In $\Tr_{H/L}(...)$ all terms in brackets except the last ones will give 0, and the latter will give -1 and there will be $p^m-1$ of them. Therefore, we will get the same result.

\subsection*{Artin-Hasse formula}
It remains to substitute into equality \ref{sec} what we got in the previous calculations and express the $\gamma$. Thus, we have obtained the Artin-Hasse formula in the case when division by isogeny gives an unramified $p-$extension withnout next $p-$roots of unity.
\begin{theorem}
Let $K=\mathbb{Q}_p$ and $F$ be a multiplicative formal group law on $K$, also let $\Pi=\zeta=\eta-1$ be a root of isogeny $[p]$ and $L=K(\zeta)$. Consider $x\in F(\mathfrak{M}_L)$ and let $y$ such that $[p](y)=x$ lies in the unramified $p-$extension $M/L$ which does not contain roots of unity of the next degree $p$, $(\Pi,x)_1=[\gamma]\zeta$. Then $$\gamma=\frac{1}{p}\Tr_{L/K}(\Pi^{-1}\eta\lambda x)$$
\end{theorem}

\end{document}